\documentclass[11pt]{amsart}
\usepackage{epsfig}

\theoremstyle{plain}
\newtheorem{theorem}{Theorem}[section]
\newtheorem{proposition}[theorem]{Proposition}
\newtheorem{lemma}[theorem]{Lemma}

\newtheorem{definition}[theorem]{Definition}

\theoremstyle{remark}
\newtheorem{remark}[theorem]{Remark}

\begin{document}

\title[Rectangle condition and $2$-fold branched cover]
{Rectangle condition for compression body and $2$-fold branched covering}

\author{Jungsoo Kim}
\address{Department of Mathematics\\
Konkuk University\\
Seoul, Korea}
\email{pibonazi@gmail.com}

\author{Jung Hoon Lee}
\address{School of Mathematics, KIAS\\
207-43, Cheongnyangni 2-dong, Dongdaemun-gu\\
Seoul, Korea}
\email{jhlee@kias.re.kr}



\begin{abstract}
We give the rectangle condition for strong irreducibility of Heegaard splittings of $3$-manifolds with non-empty boundary.
We apply this to a generalized Heegaard splitting of a $2$-fold covering of $S^3$ branched along a link.
The condition implies that any thin meridional level surface in the link complement is incompressible. We also show that the additivity of knot width holds for a composite knot satisfying the condition.
\end{abstract}

\maketitle

\section{Introduction}

A {\it compression body} $V$ is a $3$-manifold obtained from a closed surface $S$ by attaching $2$-handles to $S\times I$ on $S\times \{1\}$ and capping off any resulting $2$-sphere boundary components with $3$-balls. $S\times\{0\}$ is denoted by $\partial_+ V$ and $\partial V-\partial_+ V$ is denoted by $\partial_- V$. A {\it Heegaard splitting} $M=V\cup_S W$ is a decomposition of a $3$-manifold $M$ into two compression bodies $V$ and $W$, where $S=\partial_+V=\partial_+ W$. Every compact $3$-manifold admits Heegaard splittings.

A Heegaard splitting $V\cup_S W$ is {\it strongly irreducible} if for any essential disk $D\subset V$ and $E\subset W$, $\partial D$ intersects $\partial E$. Suppose a Heegaard splitting $V\cup_S W$ of a closed $3$-manifold is given with the information that how certain collections of essential disks of $V$ and $W$ intersect. Concerning this, there is the so-called {\bf rectangle condition}, due to Casson and Gordon, which implies that the given splitting is strongly irreducible \cite{Casson-Gordon}. See also \cite{Kobayashi}.

The paper is organized as follows.
In Section $2$, we review the rectangle condition for strong irreducibility of Heegaard splittings of closed $3$-manifolds. In Section $3$, we consider the rectangle condition for Heegaard splittings of $3$-manifolds with non-empty boundary.
In Section $4$, we introduce generalized Heegaard splitting and thin position of knots and links.
In Section $5$, we apply the rectangle condition to a generalized Heegaard splittings of a $2$-fold branched covering of $S^3$. Using the result of Scharlemann and Thompson \cite{Scharlemann-Thompson}, the condition implies that any thin surface in the $2$-fold branched covering is incompressible. Then it turns out that a thin meridional planar level surface in the link complement is also incompressible under the rectangle condition. In Section 6, we apply the rectangle condition to the additivity of knot width. If a composite knot in thin position satisfies the rectangle condition, then the width of the knot is additive with respect to connected sum.

\section{Rectangle condition: closed case}

Let $V$ be a genus $g\ge 2$ handlebody and let $S=\partial V$. Suppose a collection of $3g-3$ mutually disjoint, non-isotopic essential disks $\{D_1,\ldots,D_{3g-3}\}$ cuts $V$ into a collection of $2g-2$ balls $\{B_1,\ldots,B_{2g-2}\}$, where the shape of each $B_i$ is like a solid pair of pants. Let $P_i$ be the pair of pants $B_i\cap S$ $(i=1,\ldots,2g-2)$. Then $S=P_1\cup\ldots\cup P_{2g-2}$ is a {\it pants decomposition} of $S$.

Let $D$ be an essential disk in $V$. Assume that $D$ intersects
$\underset{i=1}{\overset{3g-3}{\bigcup}} D_i$ minimally. Hence we can see that there is no circle component in the intersection $D\cap (\underset{i=1}{\overset{3g-3}{\bigcup}} D_i)$ since $V$ is irreducible.

\begin{definition}
A {\it wave} $\alpha(D)$ for an essential disk $D$ in $V$ is a subarc of $\partial D$ cut by $\underset{i=1}{\overset{3g-3}{\bigcup}} D_i$ satisfying the following conditions.

\begin{itemize}
\item There exists an outermost arc $\beta$ and a corresponding outermost disk $\Delta$ of $D$ with $\partial\Delta=\alpha(D)\cup\beta$.
\item $(\alpha(D),\partial\alpha(D))$ is not isotopic, in $S$, into $\partial D_i$ containing $\partial\alpha(D)$.
\end{itemize}
\end{definition}

\begin{lemma}
Suppose $D$ is not isotopic to any $D_i$ $(i=1,\ldots,3g-3)$. Then $\partial D$ contains a wave.
\end{lemma}

\begin{proof}
Suppose $D\cap (\underset{i=1}{\overset{3g-3}{\bigcup}} D_i)=\emptyset$. Then $\partial D$ lives in a pair of pants $P_i$ for some $i$. Hence we can see that $D$ is isotopic to some $D_i$, which contradicts the hypothesis of the lemma. Therefore $D\cap (\underset{i=1}{\overset{3g-3}{\bigcup}} D_i)\ne \emptyset$.

Consider the intersection $D\cap (\underset{i=1}{\overset{3g-3}{\bigcup}} D_i)$. It is a collection of arcs.
Among them, there exists an outermost arc $\beta$ and corresponding outermost disk $\Delta$ with $\partial \Delta=\alpha(D)\cup\beta$ and $\alpha(D)\subset\partial D$. Since we assumed that $D$ intersects $\underset{i=1}{\overset{3g-3}{\bigcup}} D_i$ minimally, $\alpha(D)$ is not isotopic into $\partial D_i$ containing $\partial\alpha(D)$. Hence $\alpha(D)$ is a wave.
\end{proof}

Let $S$ be a closed genus $g\ge 2$ surface and $P_1$ and $P_2$ be pair of pants, which are subsurfaces of $S$ with $\partial P_i=a_i\cup b_i\cup c_i$ $(i=1,2)$. Assume that $\partial P_1$ and $\partial P_2$ intersect transversely.
For convenience, we introduce the following definition.

\begin{definition}
We say that $P_1$ and $P_2$ are tight if
\begin{itemize}
\item There is no bigon $\Delta$ in $P_1$ and $P_2$ with $\partial\Delta=\alpha\cup\beta$, where $\alpha$ is a subarc of $\partial P_1$ and $\beta$ is a subarc of $\partial P_2$.
\item For the pair $(a_1,b_1)$ and $(a_2,b_2)$, there is a rectangle $R$ embedded in $P_1$ and $P_2$ such that the interior of $R$ is disjoint from $\partial P_1\cup \partial P_2$ and the four edges of $\partial R$ are subarcs of $a_1,a_2,b_1,b_2$ respectively. Similar rectangles exist for the following combinations.
 $$(a_1,a_2,b_1,b_2)\quad (a_1,b_2,b_1,c_2)\quad (a_1,c_2,b_1,a_2)$$
 $$(b_1,a_2,c_1,b_2)\quad (b_1,b_2,c_1,c_2)\quad (b_1,c_2,c_1,a_2)$$
 $$(c_1,a_2,a_1,b_2)\quad (c_1,b_2,a_1,c_2)\quad (c_1,c_2,a_1,a_2)$$
\end{itemize}
\end{definition}

Let $V\cup_S W$ be a genus $g\ge 2$ Heegaard splitting of a
$3$-manifold $M$. Let $\{D_1,\ldots,D_{3g-3}\}$ be a
collection of essential disks of $V$ giving a pants
decomposition $P_1\cup\ldots\cup P_{2g-2}$ of $S$ and
$\{E_1,\ldots,E_{3g-3}\}$ be a collection of essential disks
of $W$ giving a pants decomposition $Q_1\cup\ldots\cup
Q_{2g-2}$ of $S$. Casson and Gordon introduced the rectangle
condition to show strong irreducibility of Heegaard splittings
\cite{Casson-Gordon}.

\begin{definition}
We say that $P_1\cup\ldots\cup P_{2g-2}$ and $Q_1\cup\ldots\cup Q_{2g-2}$ of $V\cup_S W$ satisfy the
rectangle condition if for each $i=1,\ldots,2g-2$ and
$j=1,\ldots,2g-2$, $P_i$ and $Q_j$ are tight.
\end{definition}

In Section $3$, we will give the definition of the rectangle condition for Heegaard splittings of $3$-manifolds with non-empty boundary and show that it implies strong irreducibility of the Heegaard splitting (Theorem 3.4). The proof of Theorem 3.4 is a generalization of the proof of the following proposition.

\begin{proposition}
Suppose $P_1\cup\ldots\cup P_{2g-2}$ and $Q_1\cup\ldots\cup Q_{2g-2}$ of $V\cup_S W$ satisfy the
rectangle condition. Then it is strongly irreducible.
\end{proposition}

\begin{proof}
Suppose $V\cup _S W$ is not strongly irreducible. Then there
exist essential disks $D\subset V$ and $E\subset W$ with
$D\cap E=\emptyset$. Suppose there is a bigon $\Delta$ in some $P_i$ with
$\partial\Delta=\alpha\cup\beta$, where $\alpha$ is a subarc of
$\partial D$ and $\beta$ is a subarc of $\partial P_i$. If any subarc of $\partial E$ is in $\Delta$, we
isotope it  into $S-\Delta$ across $\beta$ before we remove the bigon $\Delta$ by
isotopy of $\alpha$. So we can remove such
bigons maintaining the property that $D\cap E=\emptyset$. Also
note that the number of intersection $|\partial
E\cap (\underset{j=1}{\overset{2g-2}{\bigcup}} \partial Q_j)|$
does not increase after the isotopy. This is because there is no bigon
$\Delta'$ in $\Delta$ with $\partial\Delta'=\gamma\cup\delta$, where $\gamma$
is a subarc of $\partial P_i$ and $\delta$ is a
subarc of $\partial Q_j$ by the definition of
tightness of $P_i$ and $Q_j$. We can also remove a bigon made by a
subarc of $\partial E$ and a subarc of $\partial Q_j$
similarly. So we may assume that $D$ intersects
$\underset{i=1}{\overset{3g-3}{\bigcup}} D_i$ minimally and $E$
intersects $\underset{j=1}{\overset{3g-3}{\bigcup}} E_j$
minimally with $D\cap E=\emptyset$.

Suppose $D$ is isotopic to $D_i$ for some $i$. Let $\partial
Q_j=a_j\cup b_j\cup c_j$ ($j=1,\ldots,2g-2$). Then $\partial
D\cap Q_j$ contains all three types of essential arcs
$\alpha_{j,ab}$, $\alpha_{j,bc}$, $\alpha_{j,ca}$ by the rectangle
condition, where $\alpha_{j,ab}$ is an arc in $Q_j$ connecting
$a_j$ and $b_j$, $\alpha_{j,bc}$ is an arc connecting $b_j$ and
$c_j$ and $\alpha_{j,ca}$ is an arc connecting $c_j$ and $a_j$.
Then $E$ is not isotopic to any $E_j$ since $D\cap E=\emptyset$.
Then $\partial E$ contains a wave by Lemma 2.2 and this
contradicts that $D\cap E=\emptyset$ since a wave intersects at least one of
$\alpha_{j,ab}, \alpha_{j,bc}, \alpha_{j,ca}$ for some $j$.

If $D$ is not isotopic to any $D_i$, $\partial D$ contains a wave
by Lemma 2.2. Then also in this case, $\partial D\cap Q_j$
contains all three types of essential arcs $\alpha_{j,ab}$,
$\alpha_{j,bc}$, $\alpha_{j,ca}$ of $Q_j$ by the rectangle condition.
This gives a contradiction by the same argument as in the above.
\end{proof}

\section{Rectangle condition: bounded case}

Let $V$ a compression body and let $S=\partial_+ V$ with genus $g\ge 2$.
A {\it spanning annulus} in a compression body $V$ is an essential annulus with one boundary component in $\partial_-V$ and the other in $\partial_+ V$.
Suppose a collection of mutually disjoint, non-isotopic essential disks $\{D_i\}$ and spanning annuli $\{A_i\}$, $\{D_1,\ldots,D_k,A_{k+1},\ldots,A_{3g-3}\}$ cuts $V$ into a collection of $2g-2$ pieces $\{B_1,\ldots,B_{2g-2}\}$, where the shape of each $B_i$ is one of $(a)$, $(b)$, $(c)$ in Figure 1. Here the number $k$ is determined by the genera of the components of $\partial_-V$.

\begin{figure}[h]
   \centerline{\includegraphics[width=11cm]{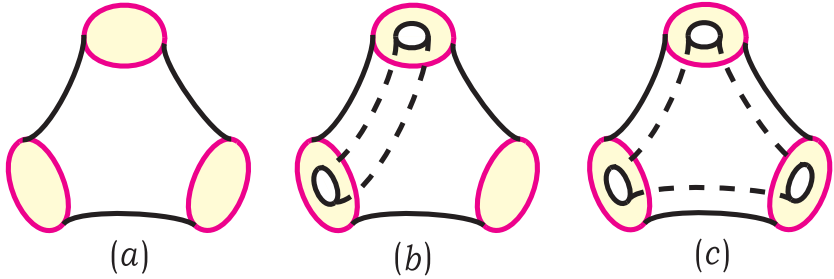}}
    \caption{Three types of $B_i$}
\end{figure}

Let $P_i$ be the pair of pants $B_i\cap S$ $(i=1,\ldots,2g-2)$. Then $S=P_1\cup\ldots\cup P_{2g-2}$ is a pants decomposition of $S$.

Using the collection $\{D_1,\ldots,D_k,A_{k+1},\ldots,A_{3g-3}\}$, we can define a wave for an essential disk in a compression body similarly as we defined a wave in a handlebody.
Note that in $(b)$ of Figure 1, wave is more restrictive. See Figure 2. In $(c)$ of Figure 1, no wave can exist.

\begin{figure}[h]
   \centerline{\includegraphics[width=9cm]{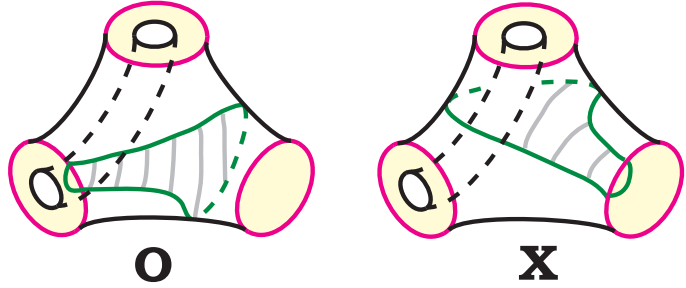}}
    \caption{Wave is more restrictive.}
\end{figure}

Let $P_1$ and $P_2$ be pair of pants, which are subsurfaces of $S$ with $\partial P_i=a_i\cup b_i\cup c_i$ $(i=1,2)$. Assume that $\partial P_1$ and $\partial P_2$ intersect transversely.
We introduce the following definition of `quasi-tight' for two pair of pants which is weaker than being tight.

\begin{remark}
Although there are many cases in the definition, the underlying idea is same.
If there is a wave $\alpha$ in $P_1$, then the rectangles give an obstruction to the existence of wave in $P_2$ that is disjoint from $\alpha$. So we can deduce a contradiction in the proof of Theorem 3.4 for the weak reducing pair of essential disks.
\end{remark}

\begin{definition}
We say that $P_1$ and $P_2$ are quasi-tight if there is no bigon $\Delta$ in $P_1$ and $P_2$ with $\partial\Delta=\alpha\cup\beta$, where $\alpha$ is a subarc of $\partial P_1$ and $\beta$ is a subarc of $\partial P_2$, and one of the following holds.

\begin{itemize}
\item Case $1$.\, Both $P_1$ and $P_2$ correspond to $(a)$ of Figure 1.

For the pair $(a_1,b_1)$ and $(a_2,b_2)$, there is a rectangle $R$ embedded in $P_1$ and $P_2$ such that the interior of $R$ is disjoint from $\partial P_1\cup \partial P_2$ and the four edges of $\partial R$ are subarcs of $a_1,a_2,b_1,b_2$ respectively. Similar rectangles exist for the following combinations.
 $$(a_1,a_2,b_1,b_2)\quad (a_1,b_2,b_1,c_2)\quad (a_1,c_2,b_1,a_2)$$
 $$(b_1,a_2,c_1,b_2)\quad (b_1,b_2,c_1,c_2)\quad (b_1,c_2,c_1,a_2)$$
 $$(c_1,a_2,a_1,b_2)\quad (c_1,b_2,a_1,c_2)\quad (c_1,c_2,a_1,a_2)$$

\item Case $2$.\, One of $P_1$ and $P_2$, say $P_1$, corresponds to $(b)$ of Figure 1 and $P_2$ corresponds to $(a)$ of Figure 1.
 Without loss of generality, assume that $a_1$ and $b_1$ are boundary components of spanning annuli. See Figure 3.

 \begin{figure}[h]
   \centerline{\includegraphics[width=8cm]{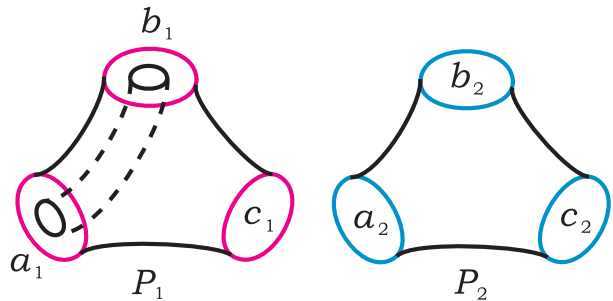}}
    \caption{$P_1$ and $P_2$}
\end{figure}

 \item Subcase $2.1$.\,
 The rectangles as above exist for the following combinations. See Figure 4.
 $$(b_1,a_2,c_1,b_2)\quad (b_1,b_2,c_1,c_2)\quad (b_1,c_2,c_1,a_2)$$
 $$(c_1,a_2,a_1,b_2)\quad (c_1,b_2,a_1,c_2)\quad (c_1,c_2,a_1,a_2)$$

  \begin{figure}[h]
   \centerline{\includegraphics[width=10cm]{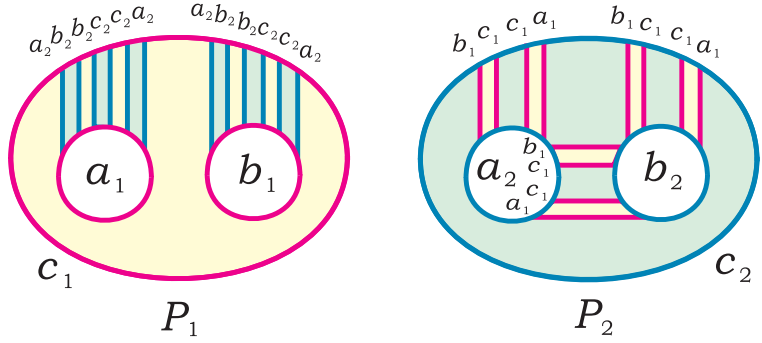}}
    \caption{Subcase $2.1$}
\end{figure}

\item Subcase $2.2$.\,
The rectangles as above exist for the following combinations. See Figure 5. There are other symmetric cases.
 $$(b_1,b_2,c_1,c_2)\quad (b_1,c_2,c_1,a_2)$$
 $$(c_1,b_2,a_1,c_2)\quad (c_1,c_2,a_1,a_2)\quad (c_1,a_2,c_1,b_2)$$

\begin{figure}[h]
\centerline{\includegraphics[width=10cm]{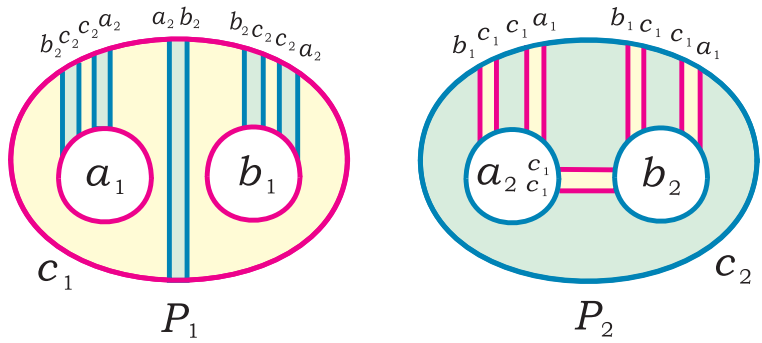}}
\caption{Subcase $2.2$}
\end{figure}

\item Subcase $2.3$.\,
The rectangles as above exist for the following combinations. See Figure 6. There are other symmetric cases.
$$(b_1,c_2,c_1,a_2)\quad (c_1,c_2,a_1,a_2)$$
$$(c_1,a_2,c_1,b_2)\quad (c_1,b_2,c_1,c_2)$$

\begin{figure}[h]
\centerline{\includegraphics[width=10cm]{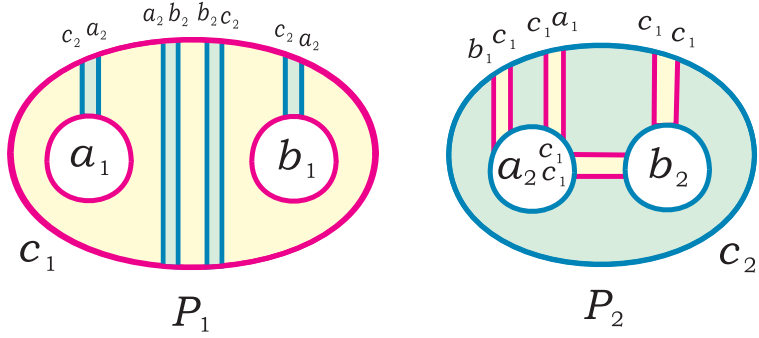}}
\caption{Subcase $2.3$}
\end{figure}

 \item Subcase $2.4$.\,
The rectangles as above exist for the following combinations. See Figure 7.
 $$(c_1,a_2,c_1,b_2)\quad (c_1,b_2,c_1,c_2)\quad (c_1,c_2,c_1,a_2)$$

   \begin{figure}[h]
   \centerline{\includegraphics[width=10cm]{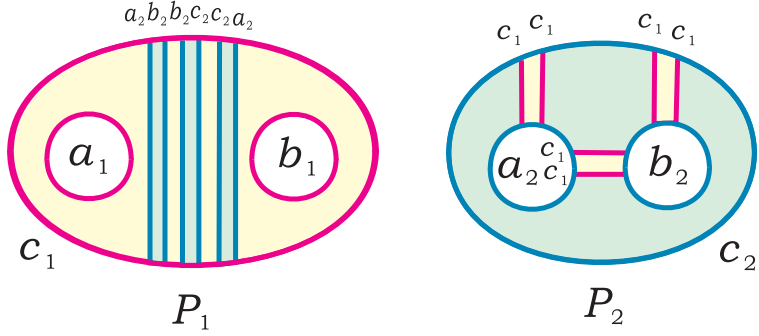}}
    \caption{Subcase $2.4$}
\end{figure}

\vspace{0.2cm}

\item Case $3$.\, Both $P_1$ and $P_2$ correspond to $(b)$ of Figure 1.
Without loss of generality, assume that $a_1$ and $b_1$ are boundary components of spanning annuli.
Also assume that $a_2$ and $b_2$ are boundary components of spanning annuli.

 \item Subcase $3.1$.\,
The rectangles as above exist for the following combinations.
 $$(b_1,b_2,c_1,c_2)\quad (b_1,c_2,c_1,a_2)$$
 $$(c_1,b_2,a_1,c_2)\quad (c_1,c_2,a_1,a_2)$$

\item Subcase $3.2$.\, The rectangles as above exist for the following combinations. There are other symmetric cases.
$$(b_1,c_2,c_1,a_2)\quad (c_1,c_2,a_1,a_2)\quad (c_1,b_2,c_1,c_2)$$

\item Subcase $3.3$.\, The rectangles as above exist for the following combinations. There is another symmetric case.
$$(c_1,b_2,c_1,c_2)\quad (c_1,c_2,c_1,a_2)$$

 \item Subcase $3.4$.\, The rectangle as above exists for the foollowing combination. $$(c_1,c_2,c_1,c_2)$$

\vspace{0.2cm}

\item Case $4$.\, $P_1$ or $P_2$ corresponds to $(c)$ of Figure 1.
In this case there is no requirement on $P_1$ and $P_2$.
\end{itemize}

\end{definition}

Let $V\cup_S W$ be a genus $g\ge 2$ Heegaard splitting of a
$3$-manifold $M$ with non-empty boundary.
Let $\{D_1,\ldots,D_k,A_{k+1},\ldots,A_{3g-3}\}$ be a
collection of essential disks and spanning annuli of $V$ giving a pants
decomposition $P_1\cup\ldots\cup P_{2g-2}$ of $S$ and
$\{E_1,\ldots,E_l,\overline{A}_{l+1},\ldots,\overline{A}_{3g-3}\}$ be a collection of essential disks and spanning annuli
of $W$ giving a pants decomposition $Q_1\cup\ldots\cup
Q_{2g-2}$ of $S$. In this setting, we give a rectangle condition for strong irreducibility.

\begin{definition}
We say that $P_1\cup\ldots\cup P_{2g-2}$ and $Q_1\cup\ldots\cup Q_{2g-2}$ of $V\cup_S W$ satisfy the
rectangle condition if for each $i=1,\ldots,2g-2$ and
$j=1,\ldots,2g-2$, $P_i$ and $Q_j$ are quasi-tight.
\end{definition}

\begin{theorem}
Suppose $P_1\cup\ldots\cup P_{2g-2}$ and $Q_1\cup\ldots\cup Q_{2g-2}$ of $V\cup_S W$ satisfy the
rectangle condition. Then it is strongly irreducible.
\end{theorem}

\begin{proof}
The idea is same with the proof of Proposition 2.5. Suppose $V\cup _S W$ is not strongly irreducible. Then there exist essential disks $D\subset V$ and $E\subset W$ with
$D\cap E=\emptyset$. First we remove bigons made by a subarc of $\partial D$ and a subarc of $\partial P_i$ by isotopy maintaining the property that $D\cap E=\emptyset$. The isotopy does not increase $|\partial
E\cap (\underset{j=1}{\overset{2g-2}{\bigcup}} \partial Q_j)|$. Similarly we remove bigons made by a subarc of $\partial E$ and a subarc of $\partial Q_j$. So we can assume that $D$ intersect the collection $\{D_1,\ldots,D_k,A_{k+1},\ldots,A_{3g-3}\}$ minimally and $E$ intersect the collection $\{E_1,\ldots,E_l,\overline{A}_{l+1},\ldots,\overline{A}_{3g-3}\}$ minimally with $D\cap E=\emptyset$.

Suppose $D$ is isotopic to $D_i$ for some $1\le i\le k$. Then by the rectangle condition we can check that there cannot exist $E$ in $W$ with $D\cap E=\emptyset$, a contradiction.

If $D$ is not isotopic to any $D_i$ $(1\le i\le k)$, $\partial D$ contains a wave by arguments as in the proof of Lemma 2.2. Then by the rectangle condition we can check that there cannot exist $E$ in $W$ with $D\cap E=\emptyset$, a contradiction.
\end{proof}

\section{Generalized Heegaard splitting and thin position of knots and links}

A {\it generalized Heegaard splitting} of a $3$-manifold $M$ is a decomposition
$$M=(V_1\cup_{S_1}W_1)\cup_{F_1}\ldots\cup_{F_{n-1}}(V_n\cup_{S_n}W_n)$$
where the collection of surfaces $\{F_i\}$ cut $M$ into submanifolds $\{M_i\}$ and $V_i\cup_{S_i}W_i$ is a Heegaard splitting of $M_i$ and $\partial M=\partial_- V_1\cup\partial_-W_n$.

A generalized Heegaard splitting is {\it strongly irreducible} if each Heegaard splitting $V_i\cup_{S_i} W_i$ is strongly irreducible. By \cite{Scharlemann-Thompson} or by Lemma 4.6 of \cite{Bachman}, if a generalized Heegaard splitting is strongly irreducible, then each surface  $F_i$ is incompressible in $M$ and it is called a {\it thin surface}.

We say that a generalized Heegaard splitting satisfies the {\it rectangle condition} if each Heegaard surface $S_i$ of $V_i\cup_{S_i} W_i$ admits a pants decomposition satisfying the rectangle condition. Then by definition and Theorem 3.4, a generalized Heegaard splitting satisfying the rectangle condition is strongly irreducible.

\begin{remark}
If we allow $2$-sphere components for $F_i$ and also allow $2$-sphere components for minus boundary components in the definition of a compression body, then a $2$-sphere component $F_i$ in a strongly irreducible generalized Heegaard splitting is essential \cite{Scharlemann-Thompson}.
\end{remark}

Let $h:S^3\rightarrow [0,1]$ be a standard height function. For a link $K$ in $S^3$, assume that  $h|_K$ is a Morse function. Let $0<c_1<\ldots<c_n<1$ be critical values of $h|_K$. Choose regular values
$0<r_1<\ldots<r_{n-1}<1$ of $h|_K$ such that $c_i<r_i<c_{i+1}$ $(1\le i\le n-1)$. The {\it width of an embedding of $K$} is $\sum_i |K\cap h^{-1}(r_i)|$. The {\it width} of $K$, denoted by $w(K)$, is the minimum taken among all embeddings of $K$. A link $K$ is in {\it thin position} if it realizes $w(K)$.

For a regular value $r$ of $h|_K$, let $S_r$ be the level sphere $h^{-1}(r)$. Let $P_r$ be the meridional planar level surface $cl(S_r-N(K))$.

\begin{definition}
An {\it upper disk} for a meridional planar level surface $P_r$ is a disk $D$ transverse to $P_r$ such that
\begin{itemize}
\item $D\subset cl(S^3-N(K))$, $\partial D=\alpha\cup\beta$, $\partial\alpha=\partial\beta$.
\item $\beta$ is an arc embedded in $\partial N(K)$, parallel to a subarc of $K$.
\item $\alpha$ is an arc properly embedded in $P_r$, and a small product neighborhood of $\alpha$ in $D$ lies above $P_r$.
\end{itemize}
\end{definition}

A {\it strict upper disk} for $P_r$ is an upper disk whose interior is disjoint from $P_r$.
A {\it lower disk}  and {\it strict lower disk} can be defined similarly.

\begin{figure}[h]
   \centerline{\includegraphics[width=7cm]{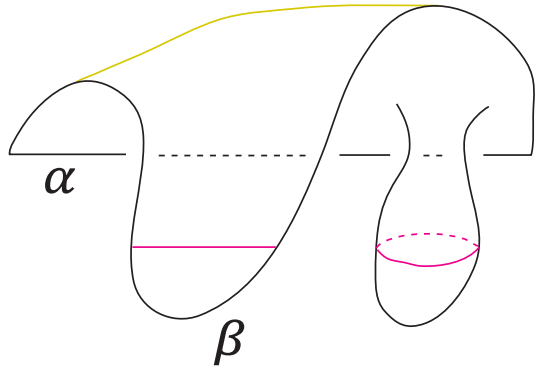}}
    \caption{An upper disk}
\end{figure}

\begin{remark}
Note that in the upper disk $D$, $\beta$ and $int(D)$ may intersect $P_r$. See Figure 8.
\end{remark}

\begin{definition}
A thin level for $K$ is a level $2$-sphere $S$ such that the following hold.
\begin{itemize}
\item $S=h^{-1}(r_i)$ for some regular value $r_i$.
\item $r_i$ lies between adjacent critical values $c_i$ and $c_{i+1}$ of $h$, where $c_i$ is a local maximum of $K$ lying below $r_i$ and $c_{i+1}$ is a local minimum of $K$ lying above $r_i$.
\end{itemize}
A thick level is a level $2$-sphere $S$ such that the following hold.
\begin{itemize}
\item $S=h^{-1}(r_i)$ for some regular value $r_i$.
\item $r_i$ lies between adjacent critical values $c_i$ and $c_{i+1}$ of $h$, where $c_i$ is a local minimum of $K$ lying below $r_i$ and $c_{i+1}$ is a local maximum of $K$ lying above $r_i$
\end{itemize}
\end{definition}

Let $S_1=h^{-1}(r_j)$ be a thin level and $S_2=h^{-1}(r_i)$ $(r_i<r_j)$ be an adjacent thick level lying below $S_1$. Consider the region $\mathcal{R}$ between $S_1$ and $S_2$. $\mathcal{R}\cap K$ consist of arcs $\{\beta_i\}$ with endpoints in $S_2$ and vertical arcs $\{\delta_i\}$. Each $\beta_i$ has exactly one local maximum. So there exists a collection of disjoint strict upper disks $\{D_i\}$ such that $\partial D_i=\beta_i\cup\alpha_i$ with $\alpha_i\subset S_2$.

\begin{figure}[h]
   \centerline{\includegraphics[width=9cm]{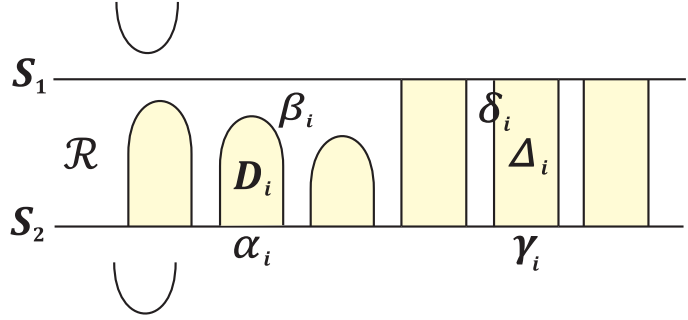}}
    \caption{A region between thin level and thick level}
\end{figure}

The number of vertical arcs is even, hence two vertical arcs can be grouped as a pair. There exists a collection of disjoint vertical rectangles $\{\Delta_i\}$ such that two vertical sides of $\Delta_i$ are a pair of $\delta_i$'s and top and bottom sides of $\Delta_i$ are in $S_1$ and $S_2$, respectively. Let $\gamma_i$ be the bottom side of $\Delta_i$. $\{\Delta_i\}$ can be taken to be disjoint from $\{D_i\}$. See Figure 9. In a region between a thick level and adjacent thin level below it, collections of disjoint strict lower disks and vertical rectangles can be taken similarly.

\section{$2$-fold branched covering}

Now we give a connection between Morse position of a link in $S^3$ and a generalized Heegaard splitting of a $2$-fold branched covering induced from the Morse position. The following construction is referred from \cite{Howards-Schultens}.

Let $\mathcal{R}$ be a region between adjacent thin level $S_1$ and thick level $S_2$ as in Section $4$.
Let $D_i$, $\alpha_i$, $\beta_i$, $\Delta_i$, $\gamma_i$, $\delta_i$ denote the same objects as in Section $4$. Cut $\mathcal{R}$ along the collections of disks $\{D_i\}$ and $\{\Delta_i\}$. Let $\mathcal{R'}$ be the resulting manifold. Take a copy of $\mathcal{R'}$, rotate it $180^{\circ}$ and attach the two copies along the corresponding $D_i$'s and $\Delta_i$'s. This is a $2$-fold covering of $\mathcal{R}$ branched along the collections $\{\beta_i\}$ and $\{\delta_i\}$. See Figure 10.

\begin{figure}[h]
   \centerline{\includegraphics[width=12cm]{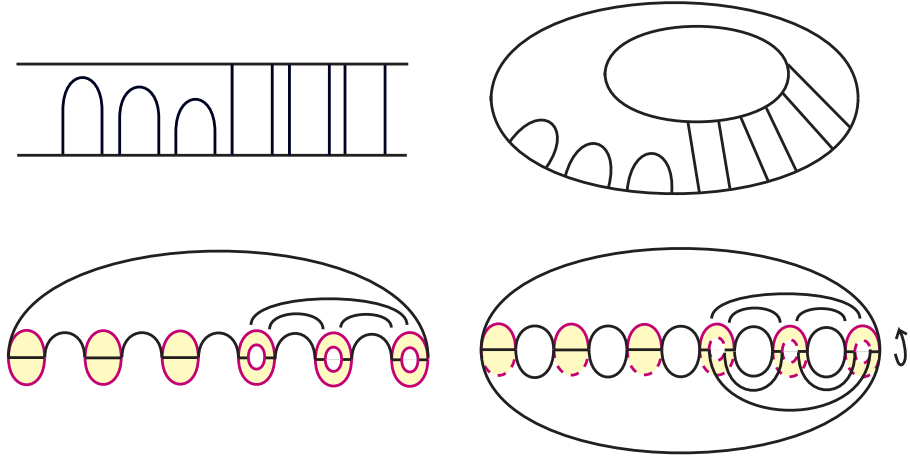}}
    \caption{A compression body obtained by $2$-fold branched covering}
\end{figure}

The resulting manifold is a compression body $V$. Let $f:V\rightarrow \mathcal{R}$ be the branched covering map. We can see that $\partial_+V=f^{-1}(S_2)$ and $\partial_-V=f^{-1}(S_1)$. $f^{-1}(D_i)$ is an essential disk in $V$ with boundary $f^{-1}(\alpha_i)$. $f^{-1}(\Delta_i)$ is a spanning annulus in $V$ with the boundary component in $\partial_+V$ being equal to $f^{-1}(\gamma_i)$.

To the collection of essential disks $\{f^{-1}(D_i)\}$ and spanning annuli $\{f^{-1}(\Delta_i)\}$, we add more essential disks and spanning annuli to give a pants decomposition of $\partial_+V$ as in $(b)$ of Figure 11. The image by $f$ of the collections of curves giving the pants decomposition is depicted in $(a)$ of Figure 11.

\begin{figure}[h]
   \centerline{\includegraphics[width=8cm]{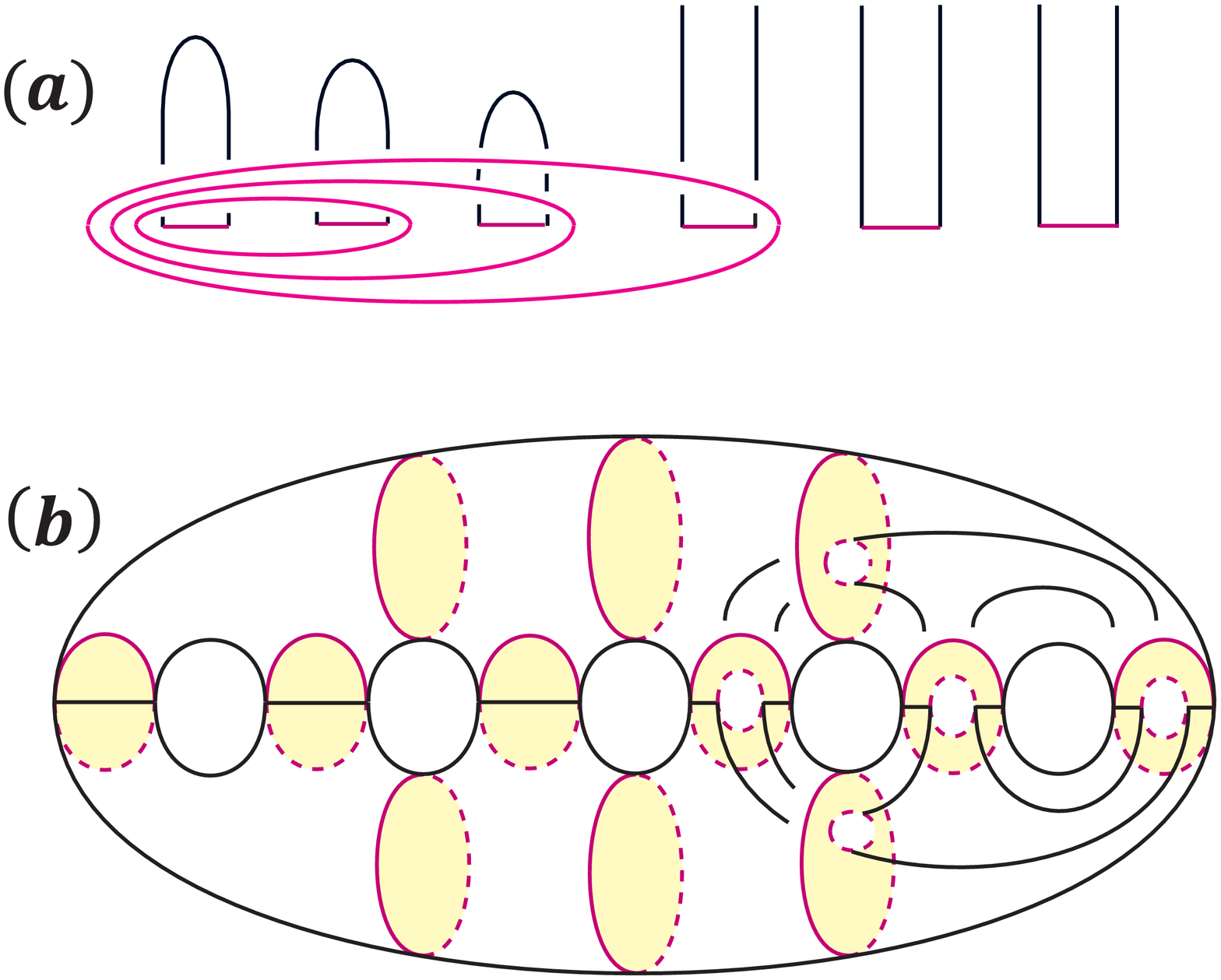}}
    \caption{Curves giving a pants decomposition and corresponding images by $f$}
\end{figure}

Similarly for other regions between adjacent thick and thin levels, we can construct a compression body by $2$-fold branched covering. Hence we get a generalized Heegaard splitting of a $2$-fold branched covering.
We can check the strong irreducibility of the generalized Heegaard splitting using the rectangle condition. We only need to check the intersection of collections of arcs and loops (as in $(a)$ of Figure 11) coming from two adjacent regions of the Morse position. In this case, we say that the link satisfies the {\it rectangle condition} if its induced generalized Heegaard splitting of the $2$-fold branched cover satisfies the rectangle condition.

\subsection{Examples}

If a link in bridge position is sufficiently complicated, it possibly satisfies the rectangle condition. Figure 12 shows an example of a $4$-bridge knot in bridge position satisfying the rectangle condition.
 Let us check it. We have the corresponding genus three Heegaard splitting of $2$-fold branched cover as in Figure 13. Let the upper handlebody be $V$ and the lower handlebody be $W$.

 \begin{figure}[h]
   \centerline{\includegraphics[width=7.5cm]{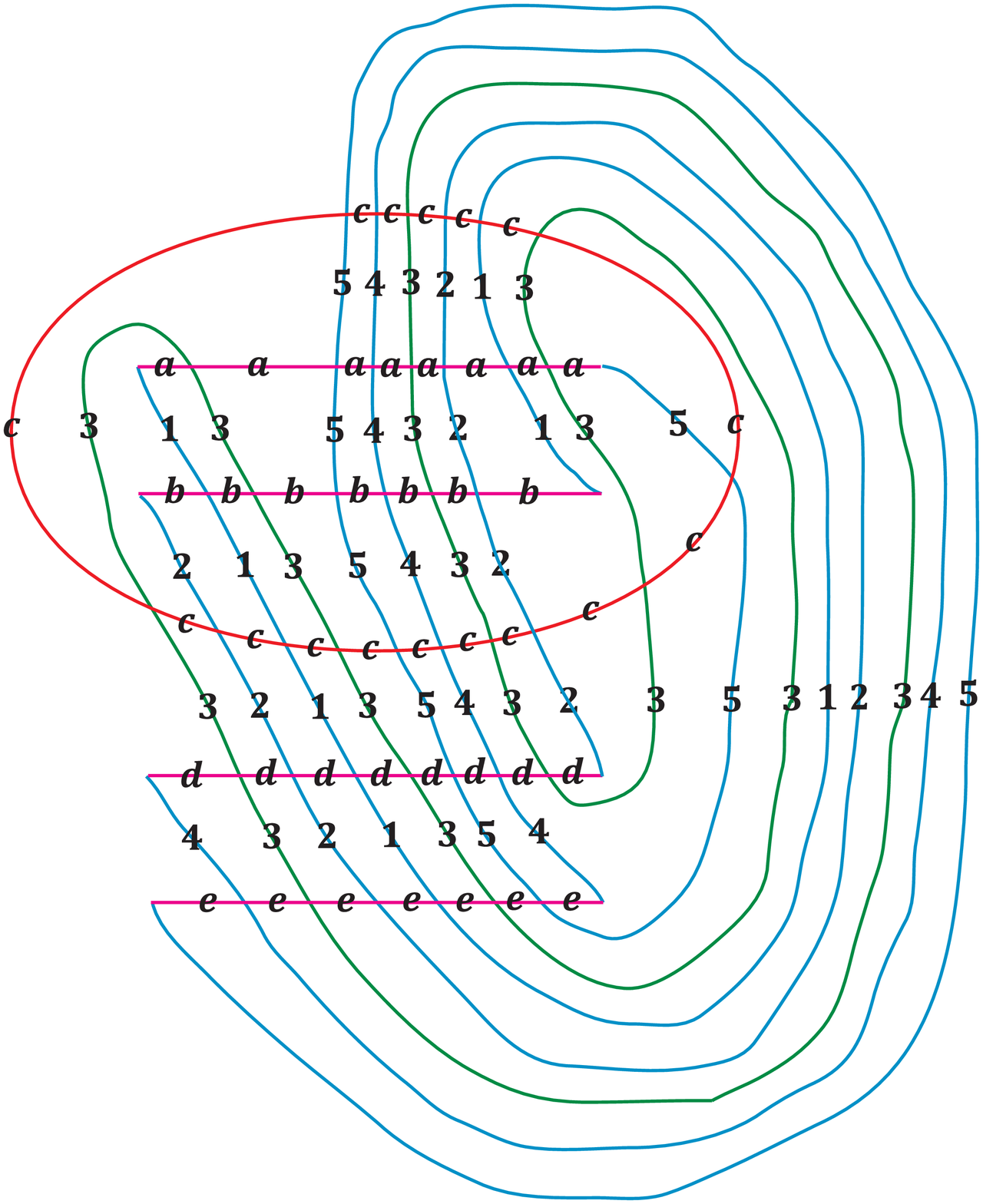}}
    \caption{An example of a $4$-bridge knot satisfying the rectangle condition}
\end{figure}

\begin{figure}[h]
   \centerline{\includegraphics[width=6cm]{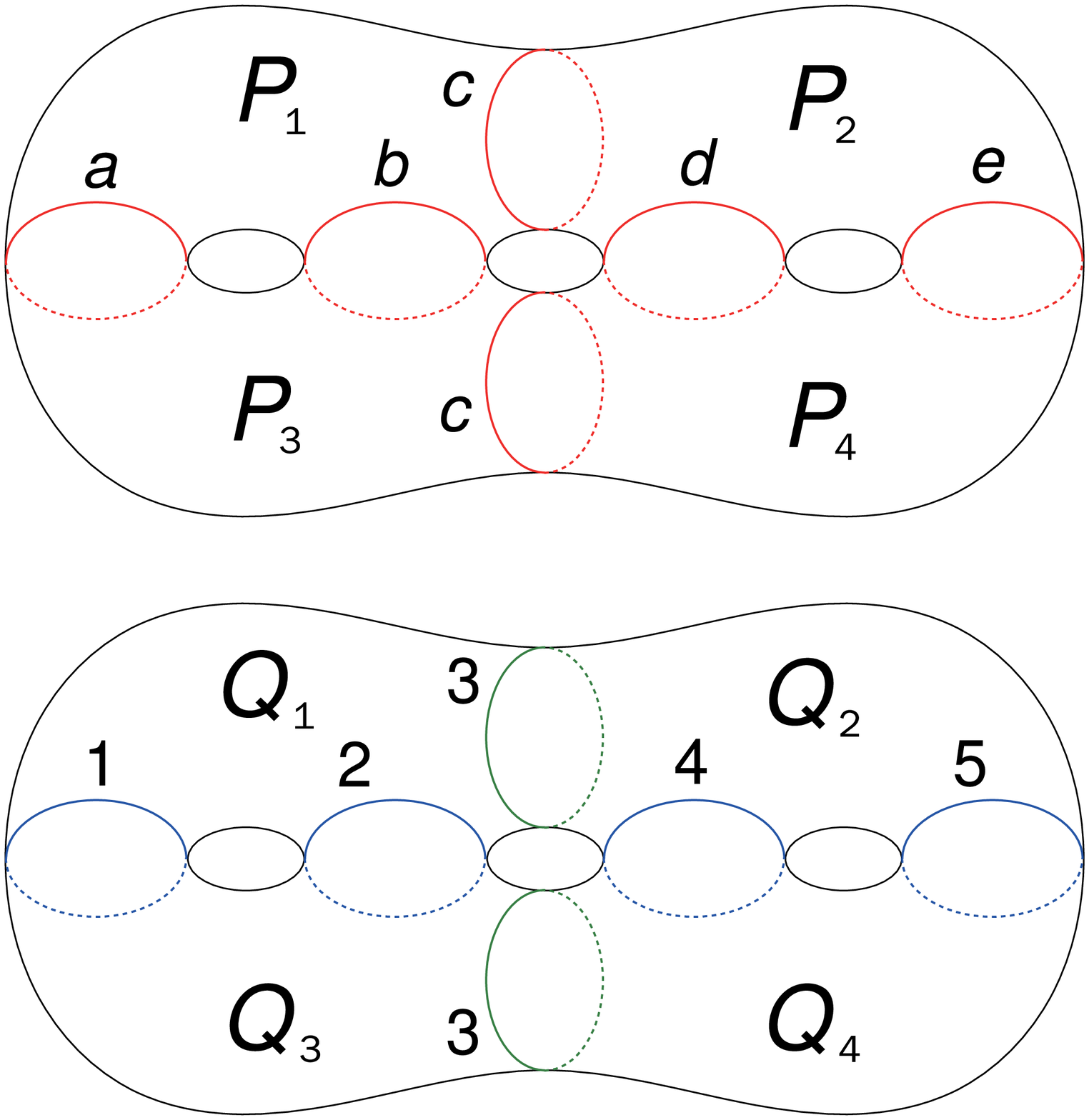}}
    \caption{}
\end{figure}

 Let $f$ be the branched covering map. Then the arcs $1,2,4,5$ of Figure 12 correspond to the image by $f$ of the loops of the same labels of Figure 13 and the loop $3$ of Figure 12 corresponds to the image by $f$ of the loop of label $3$ of Figure 13.

 To verify the rectangle condition of the splitting, let us see Figure 12. We can find the rectangle $R=(a,1,b,2)$. The preimage of $R$ for $f$ consists of two rectangles, one is in $P_1$ and the other is in $P_3$ in $\partial V$, and in a viewpoint of $\partial W$ one is in $Q_1$ and the other is in $Q_3$. Similarly, we can check for the rectangles $(a,2,b,3)$, $(a,3,b,1)$, $(b,1,c,2)$, $(b,2,c,3)$, $(b,3,c,1)$, $(c,1,a,2)$, $(c,2,a,3)$, $(c,3,a,1)$. This completes the quasi-tightness for the pairs $(P_1,Q_1)$, $(P_1,Q_3)$, $(P_3,Q_1)$, $(P_3,Q_3)$. We can prove for the pairs $\{(P_1,Q_2),(P_1,Q_4),(P_3,Q_2),(P_3,Q_4)\}$, $\{(P_2,Q_1),(P_2,Q_3),(P_4,Q_1),(P_4,Q_3)\}$, $\{(P_2,Q_2),(P_2,Q_4),(P_4,Q_2),(P_4,Q_4)\}$ similarly. This completes the proof of the rectangle condition of the $2$-fold branched cover.

\vspace{0.1cm}

Figure 14 shows an example with fewer intersections satisfying the quasi-tightness.
The right of Figure 14 means that the left diagram is on the highest thick level sphere of the knot. (For any other thick level, there exists a corresponding intersection diagram of arcs and loops like this.) Then we get the corresponding genus three Heegaard splitting of $2$-fold branched cover for the upper part of the highest thin level sphere of the knot in $S^3$ as in Figure 15. Let the upper handlebody be $V$ and the lower compression body be $W$.

\begin{figure}[h]
   \centerline{\includegraphics[width=12cm]{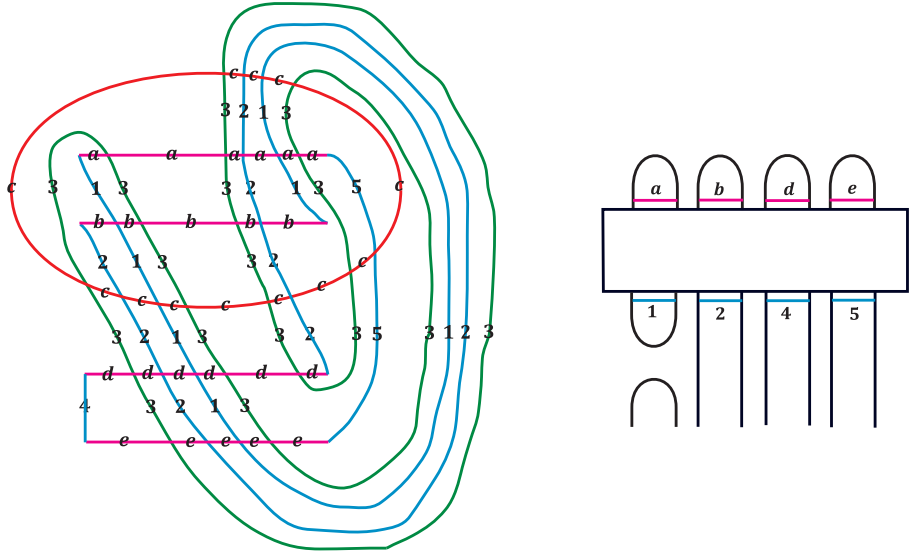}}
    \caption{An example satisfying the quasi-tightness}
\end{figure}

\begin{figure}[h]
   \centerline{\includegraphics[width=6cm]{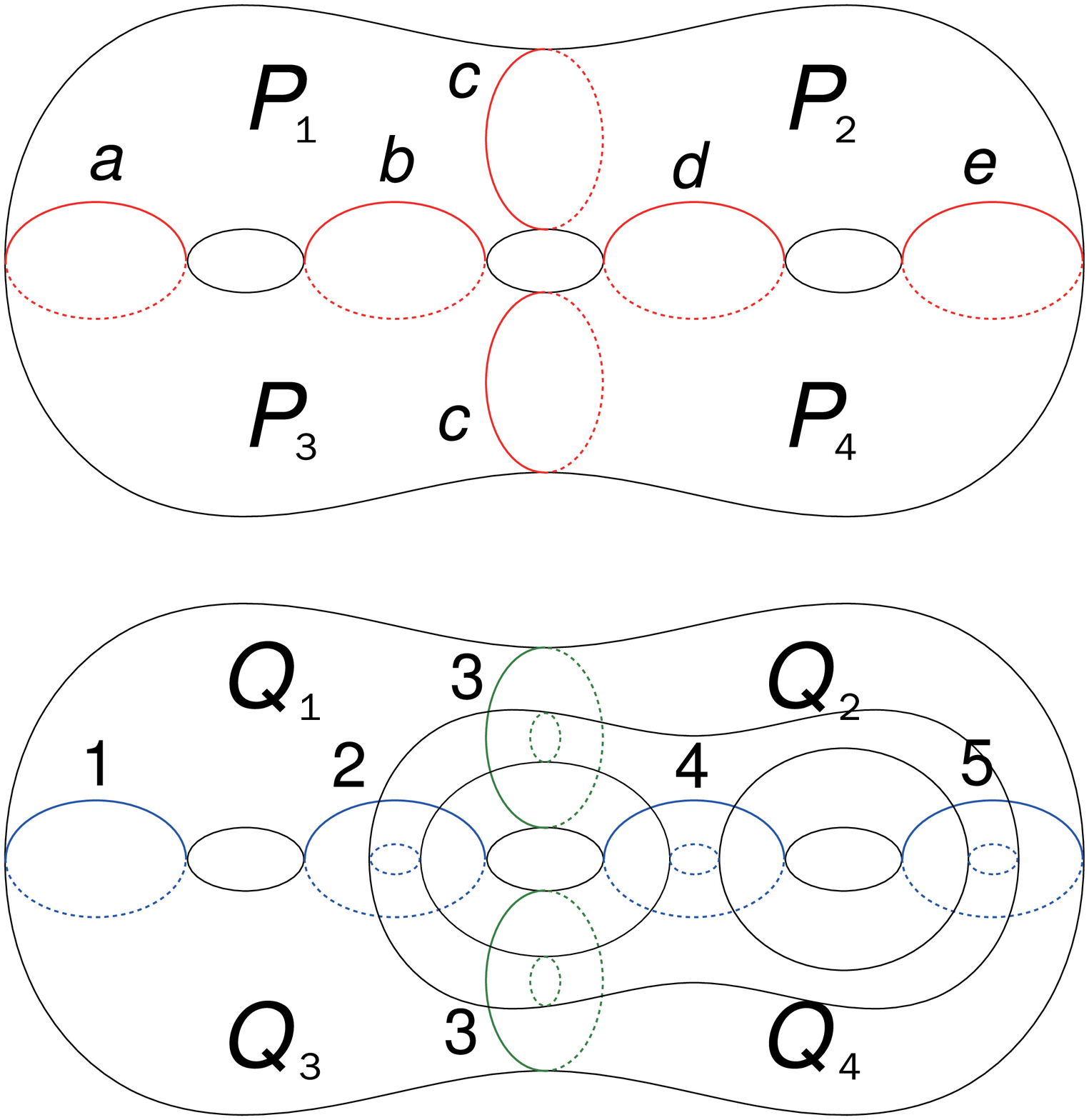}}
   \caption{}
\end{figure}

To verify the quasi-tightness of the pairs $(P_1,Q_1)$, $(P_1,Q_3)$, $(P_3,Q_1)$, $(P_3,Q_3)$, we check the existence of rectangles $(a,1,b,2)$, $(a,1,b,3)$, $(b,1,c,2)$, $(b,1,c,3)$, $(c,1,a,2)$, $(c,1,a,3)$ in Figure 14. Similarly, we can verify the quasi-tightness of the pairs $(P_2,Q_1)$, $(P_2,Q_3)$, $(P_4,Q_1)$, $(P_4,Q_3)$. Since each of both $Q_2$ and $Q_4$ corresponds to $(c)$ of Figure 1 in $W$, we don't need to check the quasi-tightness of the pairs $\{(P_1,Q_2),(P_1,Q_4),(P_3,Q_2),(P_3,Q_4)\}$, $\{(P_2,Q_2),(P_2,Q_4),(P_4,Q_2),(P_4,Q_4)\}$. This completes the proof of the rectangle condition of the $2$-fold branched cover.

We can make examples of higher bridge links satisfying the rectangle condition by similar construction.
Now the result is summarized as follows.

\begin{proposition}
The rectangle condition for the generalized Heegaard splitting of a $2$-fold branched covering can be checked by the arcs and loops as in $(a)$ of Figure 11, i.e. the rectangles formed by these arcs and loops are lifted to rectangles for the quasi-tightness of the Heegaard splitting of $2$-fold branched cover.
If it satisfies the rectangle condition, then any thin surface in the $2$-fold branched cover is incompressible.
\end{proposition}

\begin{proof}
If it satisfies the rectangle condition, the generalized Heegaard splitting is strongly irreducible.
Then by \cite{Scharlemann-Thompson}, any thin surface in the generalized Heegaard splitting is incompressible.
\end{proof}

While a thin level of lowest width is incompressible \cite{Wu}, in general a thin level surface can possibly be compressible in the link complement \cite{Tomova}. We have the following.

\begin{theorem}
The rectangle condition implies that every thin meridional planar level surface is incompressible in the link complement.
\end{theorem}

\begin{proof}
Let $S$ be a thin level surface for the link $K$. Suppose $P=cl(S-N(K))$ is compressible and let $D$ be a compressing disk for $P$. $\partial D$ decomposes $S$ into two disks $\Delta_1$ and $\Delta_2$. Since $D$ is a compressing disk, both $\Delta_1$ and $\Delta_2$ intersects $K$, in even number of points. Let $f:M\rightarrow S^3$ be the branched covering map. Then $f^{-1}(D)$ is a disjoint union of two disks $E_1\dot{\cup}\, E_2$, where $E_i\cap f^{-1}(S)=\partial E_i$ $(i=1,2)$ and $\partial E_i$ is essential in $f^{-1}(S)$. ($\partial E_1\cup \partial E_2$ cuts $f^{-1}(S)$ into two non-disk components $f^{-1}(\Delta_1)$ and $f^{-1}(\Delta_2)$.) This means that the thin surface $f^{-1}(S)$ is compressible in $M$, which contradicts Proposition 5.1.
\end{proof}

\section{Application to additivity of knot width}

Let $K=K_1\# K_2$ be a composite knot. Put a thin position of $K_1$ vertically over a thin position of $K_2$ and do a connected sum operation by vertical arcs. This gives a presentation of $K$ from which we can see that $w(K)\le w(K_1)+w(K_2)-2$. There are cases where equality holds, that is, width is additive under connected sum.
Rieck and Sedgwick showed that it is true for small knots \cite{Rieck-Sedgwick}. However, recently Blair and Tomova showed that width is not additive in general \cite{Blair-Tomova}.

In \cite{JungsooKim}, the first author gave some condition for knots satisfying the additivity of width via $2$-fold branched covering as discussed in Section $5$.

\begin{theorem}[\cite{JungsooKim}]
If a thin position of a knot $K$ induces a strongly irreducible generalized Heegaard splitting of the $2$-fold branched cover of $(S^3,K)$, then the thin position of $K$ is the connect sum of thin position of prime summands of $K$ vertically. Therefore $K$ satisfies the additivity of knot width.
\end{theorem}

Combining the results in Section $5$ and Theorem 6.1, we have the following.

\begin{theorem}
If a knot $K$ is in thin position and satisfies the rectangle condition, then $w(K)=w(K_1)+\ldots+w(K_n)-2(n-1)$ for prime summands $K_1,\ldots,K_n$ of $K$.
\end{theorem}

\end{document}